\documentclass[11pt]{amsart}
\usepackage{a4wide}
\usepackage{latexsym,bm,amsmath,amssymb}

\newtheorem{thm}{Theorem}[section]
\newtheorem{cor}[thm]{Corollary}
\newtheorem{lem}[thm]{Lemma}

\theoremstyle{definition}

\newtheorem{defn}[thm]{Definition}

\newtheorem{rem}[thm]{Remark}
\numberwithin{equation}{section}

\usepackage[colorlinks, linkcolor=blue, citecolor=red, urlcolor=black]{hyperref}

\def\bysame{\leavevmode\hbox to3em{\hrulefill}\thinspace}

\newcommand{\osc}{\operatorname{osc}}

\begin{document}

\title[A priori estimates for
parabolic Monge-Amp\`ere type equations]{A priori estimates for \\
parabolic Monge-Amp\`ere type equations}

\author[Y. Zhou and R. Zhu]{Yang Zhou and Ruixuan Zhu}

\address[Yang Zhou]{The Institute of Mathematical Sciences, The Chinese University of Hong Kong, Shatin, NT, Hong Kong}
\email{yzhou@ims.cuhk.edu.hk}

\address[Ruixuan Zhu]{Institute for Theoretical Sciences, Westlake University, Hangzhou, 310030, China}
\email{zhuruixuan@westlake.edu.cn}



\begin{abstract}
We prove the existence and regularity of convex solutions
to the first initial-boundary value problem for the parabolic Monge-Amp\`ere equation
\begin{equation*}
  \left\{\begin{aligned}
    -u_t + \det D^2u &= \psi(x,t) \quad\quad\ \text{ in } Q_T,\\
    u&=\phi\quad\quad \quad\quad \, \  \text{ on } \partial_pQ_T,
  \end{aligned}\right.
\end{equation*}
where $\psi,\phi$ are given functions, $Q_T=\Omega\times(0,T]$,
$\partial_p Q_T$ is the parabolic boundary of $Q_T$,
and $\Omega\subset\mathbb{R}^n$ is a uniformly convex domain.
Our approach can also be used to prove similar results for the $\gamma$-Gauss curvature flow with any $0<\gamma\le 1$.
\end{abstract}
\maketitle

\section{Introduction}

The Monge-Ampère equation arises in various geometric problems, including the Minkowski problem and geometric flows. The elliptic Monge-Ampère equations have been extensively studied since the 1970s; we refer readers to~\cite{figalli,twbook} and references therein for further details.

Over the past decades, the following parabolic Monge-Amp\`ere equation has been widely investigated:
\begin{equation}\label{eq:1.1}
  \left\{\begin{aligned}
    -u_t + \rho(\det D^2u) &= \psi(x,t) \quad\quad\ \text{ in } Q_T,\\
    u&=\phi\quad\quad \quad\quad \, \  \text{ on } \partial_pQ_T,
  \end{aligned}\right.
\end{equation}
where $Q_T=\Omega\times(0,T]$, $\partial_p Q_T$ is the parabolic boundary of $Q_T$,
and $\Omega$ is a convex domain in $\mathbb{R}^n$.

To our best knowledge, the literature on this equation can be categorized into three main cases:
\begin{itemize}
    \item[$(i)$] $\rho(s) = s^{\frac1n}$.
    \item[$(ii)$] $\rho(s) = - 1/s$ and $\psi\equiv 0$;
    \item[$(iii)$] $\rho(s) = \log s$.
\end{itemize}

In case $(i)$, Ivochkina and Ladyzhenskaia~\cite{I} obtained the \textit{a priori} $\tilde{C}^{2,\alpha}$ estimates of the first initial-boundary value problem (see also~\cite{lie}) if $ \Omega$ and $\phi(\cdot, 0)$ are uniformly convex, and if either
\begin{equation}
 \inf_{\partial_pQ_T} \phi_t + \inf_{Q_T} \psi > \frac{1}{2}  ad^2
\end{equation}
or
\begin{equation}\label{eq:il}
\inf_{\partial_pQ_T} ( \phi_t+\psi) > 0 \quad \text{and} \quad
\psi  \text{ is ~ concave~ with~ respect~ to~ $x$,}
\end{equation}
where $a = \max\{\sup_{Q_T} \psi_t, 0\}$ and $d = \text{diam}\, \Omega$. 
The \textit{a priori} $C^{1,\alpha}$ and $W^{2,p}$ estimates for a generalized form of this equation were established in~\cite{hl, ds2012}. Boundary regularity results were subsequently obtained by Tang~\cite{tang}.
Caffarelli-Schauder type estimates were derived in~\cite{ww2} under the assumption that $\psi$ is Lipschitz continuous. 

For case $(ii)$, the equation was used by Tso~\cite{tso} to investigate the Gauss curvature flow, where he derived $\tilde{C}^{2,\alpha}$ \textit{a priori} estimates. The Schauder estimates and $W^{2,p}$ estimates were later established in~\cite{ww1} and~\cite{GH2001}, respectively. The J\"orgens-Calabi-Pogorelov type results on entire solutions were obtained in \cite{GH1998,zb}.
Moreover, one can refer to \cite{lie,kr1976} for another type of parabolic Monge-Amp\`ere equation $-u_t \det D^2u = \Psi$. 

In case $(iii)$, Tso \cite{tso2} studied the first initial-boundary value problem and derived the $\tilde{C}^{2,\alpha}$ \textit{a priori} estimates. 
Xiong and Bao~\cite{xiong} considered a generalized form where $\rho(s)=\eta(\log s)$ satisfies $\eta'>0$ and $ n\eta'' \le \eta'$ to make the parabolic operator 
$$P(-u_t, D^2u):=-u_t + \rho(\det D^2u)$$ 
concave with respect to $D^2u$. Schnürer and Smoczyk investigated the Neumann boundary problem for a similar equation in~\cite{sh}.

However, the seemingly simplest scenario $\rho(s)=s$ remains less understood. 
Unlike the three cases discussed above, the parabolic operator $P(-u_t, D^2u)= -u_t + \det D^2u$ is not concave.
The existence of a global smooth solution of \eqref{eq:1.1} in this case is a long-standing open problem due to the lack of concavity.

In a celebrated work~\cite{cns1}, Caffarelli, Nirenberg, and Spruck established global $C^{2,\alpha}$ estimates for the Dirichlet problem of elliptic Monge-Amp\`ere equations $\det D^2u=\psi$, proving the existence of unique smooth convex solutions. 
This naturally raises the question of whether the first initial-boundary value problem of parabolic Monge-Amp\`ere equation 
\begin{equation}\label{1-eq}
  \left\{\begin{aligned}
    -u_t + \det D^2u &= \psi(x,t) \quad\quad\ \text{ in } Q_T,\\
    u&=\phi\quad\quad \quad\quad \, \  \text{ on } \partial_pQ_T,
  \end{aligned}\right.
\end{equation}
has such global regularity under suitable conditions.
In this article, we provide an affirmative answer to this question under conditions \eqref{eq:il} given by Ivochkina and Ladyzhenskaia~\cite{I}.
Specifically, we assume that
$\Omega$ is a uniformly convex domain in $\mathbb{R}^n$ with smooth boundary, and $\psi,\phi$ are smooth functions satisfying the following conditions:
\begin{itemize}
    \item[$(\mathcal P1)$] $\phi_t + \psi \ge c_0>0$ on $\partial\Omega\times[0,T]$;
    \item[$(\mathcal P2)$] $\phi(x,0)$ is uniformly convex with respect to $x$; 
    \item[$(\mathcal P3)$] $\psi(x,t)$ is concave with respect to $x$.
\end{itemize}

In this paper, we consider solutions in the following Banach space.
\begin{defn}\label{hs}
For $Q=\Omega\times I$, where $\Omega\subset\mathbb{R}^n$ is an open convex domain, and $I=(a,b]$,
we denote by $\tilde{C}^{k,\alpha}(\overline{Q})$ the Banach space with the norm
{\small \begin{equation*}
  \begin{aligned}
    ||w||_{\tilde{C}^{k,\alpha}(\overline{Q})}=\sup _{\substack{|\gamma|+2 s \leqq k \\(x, t) \in Q}}\left|D_x^\gamma D_t^s w(x, t)\right| +\sup _{\substack{|\gamma|+2 s=k \\(x, t),(y, \tau) \in Q}} \frac{\left|D_x^\gamma D_t^s w(x, t)-D_y^\gamma D_{\tau}^s w(y, \tau)\right|}{\left(|x-y|^2+|t-\tau|\right)^{\alpha / 2}}.
  \end{aligned}
\end{equation*}}
\end{defn}

We shall prove
\begin{thm}\label{0-pri}
Let $u\in \tilde{C}^4(\overline{Q}_T)$ be convex in $x$ and solve equation~\eqref{1-eq} under the above assumptions.
Then it holds
$$||u||_{\tilde{C}^{2,\alpha}(\overline{Q}_T)}\le C,$$
for some positive constants $C$ and $\alpha\in (0, 1)$
depending only on $n,T,\Omega, c_0,||\psi||_{\tilde{C}^2(\overline{Q}_T)},
||\phi||_{\tilde{C}^4(\overline{Q}_T)}$ and the uniform convexity of $\phi(\cdot,0)$.
\end{thm}

By the \textit{a priori} estimate, we obtain the existence of solution to equation~\eqref{1-eq}. 

\begin{thm}\label{thm1.1}
Assume that equation~\eqref{1-eq} satisfies the compatibility condition of any order (see Remark~\ref{rem}) and the above assumptions. Then there exists a unique smooth solution $u$, convex in $x$, to the first initial-boundary value problem~\eqref{1-eq}.
\end{thm}

\begin{rem}
    It is worth pointing out that in subsection \ref{s3.1}, we take examples to show that condition $({\mathcal P3})$ is necessary for the existence of convex solutions to problem~\eqref{1-eq}.
\end{rem}

We can also prove Theorems~\ref{0-pri} and~\ref{thm1.1} for Gauss curvature flow with boundary
\begin{equation}\label{1-gcf}
  \left\{\begin{aligned}
    u_t&=\frac{\det D^2u}{(1+|Du|^2)^{(n+1)/2}}\quad\quad \text{ in } Q_T,\\
    u&=\phi\quad\quad\quad\quad\quad\quad\quad\quad\quad\;  \text{ on } \partial_pQ_T.
  \end{aligned}\right.
\end{equation} 
Equation~\eqref{1-gcf}, corresponding to this setting, was studied in~\cite{O}, where the author proved the existence of certain self-similar solutions and obtained gradient estimates in the homogeneous case $\phi\equiv 0$. However, higher-order global regularity for this case remains an open problem.

Our main difficulty in proving Theorem \ref{0-pri} lies in proving the interior second-order derivative estimates and the boundary $\tilde C^{2,\alpha}$ estimates under the lack of concavity.
It is worth pointing out that our approach to obtaining such global estimates requires a combination of corresponding boundary estimates for equation \eqref{1-eq} and interior estimates for equation \eqref{duality}, the equation of the Legendre transform of $u$, rather than studying only one of the equations to obtain them. 

The proof of Theorem~\ref{0-pri} can be divided into the following three steps:

{\bf Step 1:~ Establish the $C^2$ estimate on the boundary}

We first obtain the $C^0$ and $C^1$ estimates, and
\begin{equation}\label{0-es-t}
 C^{-1}\le u_t+\psi\le C,
\end{equation}
for some positive constant $C$.
Then, we obtain the bound of $D^2u$ on the parabolic boundary by using the method in~\cite{cns1}.
Thus, by~\eqref{0-es-t} we have the estimates for the eigenvalues
\begin{equation}\label{0-es-u}
\lambda(D^2u)\ge C\quad \text{on}\quad \partial_p Q_T.
\end{equation}

{\bf Step 2:~ Establish the global $C^2$ estimate by the duality method}

Due to the lack of concavity, we cannot use the maximum principle to control the global maximum of $|D^2u|$ by $\sup_{\partial_p Q_T} |D^2u|$. Instead, we apply the estimates in {\bf Step 1} to the equation after the Legendre transform.

Let $U$ be the Legendre transform of $u$, given by
$$U(y,t)=\sup \{ y\cdot x-u(x,t):\ x\in \Omega\}. $$
Denote $Q_T^*=\{(y,t):\, y\in Du(\Omega,t),t\in(0,T]\}$.
Then $U$ satisfies
\begin{equation}\label{duality}
  -U_t - \frac1{\det D^2U} = -\psi(DU,t) \quad\text{ in }Q_T^*.
\end{equation}
Since the above equation is concave, one can infer
$$\sup_{Q_T^*}||D^2U||\le \sup_{\partial_pQ_T^*}||D^2U||,$$
which combined with~\eqref{0-es-t} and~\eqref{0-es-u} yields an upper bound for the eigenvalues
$\lambda(D^2U)$ in $Q^*_T.$
Then, by duality we have
$\lambda(D^2u)\ge C~ \text{in}~Q_T.$
By~\eqref{0-es-t}, it follows that
\begin{equation}
C^{-1}\le \lambda(D^2u)\le C \quad \text{in}\quad Q_T.
\end{equation}

{\bf Step 3:~ Establish the global $\tilde C^{2,\alpha}$ estimate}

According to {\bf Step 2}, equation~\eqref{duality} is uniformly parabolic.
Since the equation is also concave, we can obtain the space H\"older estimate for $D^2U(\cdot,t)$ for each $t$
by the Evans-Krylov theorem.
The interior $\tilde C^{2,\alpha}$ regularity then follows by~\cite{E,Kr}
and the interior higher-order regularity of $U$ also follows by the bootstrap technique.

As for the boundary estimate, we first use Krylov-Safonov's estimate (see~\cite{KS})
to get the global space-time H\"older estimate for $u_t$.
Then by differentiating the equations along tangential directions $\tau_k$,
we shall prove the $\tilde C^{1,\alpha}$ regularity of $D_{\tau_k}u$ on the lateral surface of the parabolic boundary,
followed by the space-time H\"older estimate for second normal derivatives.

This paper is arranged as follows.
In Section~\ref{s2}, we establish the {\it a priori} estimates in $\tilde C^{2,\alpha}(\Omega\times[\epsilon,T])$.
We then obtain the existence and uniqueness of solutions in Section~\ref{s3},
completing the proof of Theorems~\ref{0-pri} and~\ref{thm1.1}.
Finally, we prove the existence of a unique smooth solution to equation~\eqref{1-gcf} in Section~\ref{s4}.

\section{The a priori estimate}\label{s2}
We first consider the parabolic Monge-Amp\`ere equation
\begin{equation}\label{eq}
  \left\{\begin{aligned}
    -u_t + \det D^2u &= \psi(x,t) \quad\quad\ \text{ in } Q_T,\\
    u&=\phi\quad\quad \quad\quad \, \  \text{ on } \partial_pQ_T,
  \end{aligned}\right.
\end{equation}
where $Q_T=\Omega\times(0,T]$, $\partial_p Q_T$ is the parabolic boundary of $Q_T$, and $\Omega$ is a uniformly convex domain in $\mathbb{R}^n$ with smooth boundary.
Assume that $\psi,\phi$ are smooth functions satisfying assumptions $(\mathcal P1)- (\mathcal P3)$.

In this section, we suppose that $u\in \tilde{C}^{4}(\overline{Q}_T)$ is convex in $x$ and a solution to equation~\eqref{eq}. We aim to prove the $\tilde{C}^{2,\alpha}$ {\it a priori} estimate for $u$.
For this purpose, we shall obtain the $C^2$ estimate for $u$ on the parabolic boundary first.

\begin{lem}[Comparison Principle] \label{le:cp}
    Assume that $w,v\in\tilde{C}^{2}(\overline{Q}_T \setminus \partial_p Q_T) \cap C(\overline{Q}_T)$ and one of $w,v$ is convex in $x$.
    If $-w_t + \det D^2w \ge -v_t + \det D^2v$ in $\overline{Q}_T \setminus \partial_p Q_T$ and $w \le v$ on $\partial_p Q_T$, then $w \le v$ in $Q_T$.
\end{lem}

\begin{proof}
    See \cite[Theorem 14.1]{lie}.
\end{proof}

\subsection{Step 1: $C^2$ estimates on the boundary}
\hfill

At first, we obtain the $C^0$ and $C^1$ estimates, stated as the lemma below.
\begin{lem}\label{lem1}
$u(\cdot,t)$ is strictly convex for any $t\in[0,T]$, and the following estimates hold,
  \begin{equation}\label{es-01}
    ||u||_{C^1(\overline{Q}_T)}\le C,
  \end{equation}
  and
  \begin{equation}\label{es-t}
    1/C\le u_t+ \psi\le C.
  \end{equation}
  Here $C$ is a universal positive constant that depends only on $n, T$, $\Omega$, $c_0$, $||\phi||_{\tilde C^2(\overline{Q}_T)}$, $\|\psi\|_{L^\infty({Q}_T)}$, $\|\psi_t\|_{L^{n+1}({Q}_T)}$ and the uniform convexity of $\phi(\cdot,0)$.
\end{lem}

\begin{proof}
By the comparison principle, we have
\begin{equation*}
    \pm u \le \Big(\sup_{\partial_p Q_T} |\phi| + \sup_{{Q}_T}|\psi| \Big) e^t \quad \text{in} \  Q_T,
\end{equation*}
which yields that $\|u\|_{C^0(\overline{Q}_T)} \le (\sup_{\partial_p Q_T} |\phi| + \sup_{{Q}_T}|\psi|) e^T$.

To obtain \eqref{es-t}, differentiating the equation \eqref{eq} with respect to $t$, we have
\begin{equation}\label{eq:l-ut}
  \mathcal L u_t =\psi_t \quad\text{ in } \ Q_T,  
\end{equation}
where the linearized operator $\mathcal L$ of \eqref{eq} is
\begin{equation}\label{eq:l}
    \mathcal L  = - \partial_t + \det D^2 u \cdot u^{ij} \partial_{ij},
\end{equation}
and $u^{ij}$ is the $(i,j)$-entry of the inverse matrix of $D^2u$. Since $\psi(x,t)$ is concave in $x$, then
\begin{equation*}
    \mathcal L(u_t +\psi)  = \det D^2 u \cdot u^{ij} \partial_{ij}\psi \le 0.
\end{equation*}
It follows form the weak maximum principle that
\begin{equation*}
    \inf_{Q_T} (u_t +\psi) = \inf_{\partial_p Q_T} (u_t +\psi) = \min \{ \inf_{\partial\Omega\times(0,T]} (\phi_t +\psi), \inf_{\Omega\times\{0\}} (\det D^2 \phi) \}.
\end{equation*}
That is
\begin{equation}\label{eq:ut-lb}
    \inf_{Q_T} (u_t +\psi)  \ge \min \{ c_0,  \inf_{\Omega} (\det D^2 \phi(\cdot, 0)) \}>0,
\end{equation}
or,
\begin{equation*}
    \inf_{Q_T} u_t \ge - \sup_{Q_T} |\psi| + \min \{ c_0,  \inf_{\Omega} (\det D^2 \phi(\cdot, 0) \}.
\end{equation*}

By using the parabolic Alexandrov-Bakelman-Pucci maximum principle (\cite[Theorem 7.1]{lie}) to equation \eqref{eq:l-ut}, we have
\begin{equation}
    \sup_{{Q}_T} u_t \le \sup_{\partial_p{Q}_T} u_t + C_n (\text{diam}\,\Omega)^{\frac{n}{n+1}} \left(\iint_{\Gamma_{u_t}} |u_{tt} \det D^2 u_t| dxdt \right)^{\frac{1}{n+1}},
\end{equation}
where
\begin{equation*}
    \Gamma_{u_t} = \{(x,t)\in Q_T : u_{tt} \ge 0, D^2 u_t  \le 0\}, \quad C_n = (n(n+1)/\omega_n)^{\frac{1}{n+1}},
\end{equation*}
and $\omega_n$ is the surface area of unit sphere in $\mathbb R^n$.
Note that on $\Gamma_{u_t}$,
 \begin{equation*}
 \begin{split}
     u_{tt} \det (-D^2 u_t) & = (\det D^2u)^{-\frac{n-1}{n+1}} u_{tt} \cdot \det \Big[ (\det D^2u)^{\frac{2}{n+1}} (D^2 u)^{-1} (-D^2 u_t) \Big] \\
     & \le \left( \frac{ u_{tt} + (\det D^2u) \text{tr} \Big[  (D^2 u)^{-1} (-D^2 u_t) \Big] }{(n+1)(\det D^2u)^{\frac{n-1}{n+1}}} \right)^{n+1} \\
     & = \left( \frac{ - \mathcal L u_t }{(n+1)(u_t +\psi)^{\frac{n-1}{n+1}}} \right)^{n+1}.
 \end{split} 
 \end{equation*}

It follows that
\begin{equation}\label{eq:ut-ub}
\begin{split}
    \sup_{{Q}_T} u_t & \le \sup_{\partial_p{Q}_T} u_t + \frac{C_n (\text{diam}\,\Omega)^{\frac{n}{n+1}}}{(n+1) (\inf_{Q_T} (u_t +\psi))^{\frac{n-1}{n+1}} }   \left(\iint_{\Gamma_{u_t}} \big[(-\psi_t)^+ \big]^{n+1} dxdt \right)^{\frac{1}{n+1}} \\
    & \le \max\{  \sup_{\partial\Omega\times(0,T]} \phi_t, \sup_{\Omega\times\{0\}} (\det D^2 \phi -\psi) \} + \frac{C_n (\text{diam}\,\Omega)^{\frac{n}{n+1}} \|\psi_t\|_{L^{n+1}(Q_T)} }{(n+1) (\inf_{Q_T} (u_t +\psi))^{\frac{n-1}{n+1}} }.
\end{split}
\end{equation}
Therefore, we have proved~\eqref{es-t} according to \eqref{eq:ut-lb} and \eqref{eq:ut-ub}.
The strict convexity follows directly from \eqref{es-t} and the smoothness of $\phi$, see~\cite{ca1} for a reference. 

To estimate $|Du|$, it suffices to prove the bound for $|D_{\nu}u|$ on $\partial\Omega\times[0,T]$ for the outward unit normal vector $\nu$, since $|Du|$ is bounded by $|D\phi|$ on $\Omega\times \{0\}$ and tangentially on $\partial\Omega\times[0,T]$.
For any fixed $t_0\in [0,T]$, take a smooth auxiliary function $w(x)$ such that $w=\phi(t_0)$ on $\partial\Omega$, and $\det D^2w\ge \sup_{Q_T}(u_t + \psi)$ in $\Omega$. 
Note that such $w$ does exist, and one example is to assume $w$ with $(-d_{\partial\Omega}+d^2_{\partial\Omega})\delta+\phi(t_0)$ near the boundary for some constant $\delta>0$ large enough.
Then by the classical comparison principle for Monge-Amp\`ere equations, we have $u\ge w$ in $\Omega\times\{t_0\}$.
Thus $D_{\nu} u\le D_{\nu}w\le C$.
Till now we have proved formula~\eqref{es-01}.
\end{proof}

The proof for the second-order derivative estimates on the parabolic boundary follows a similar structure as in Section 2 of~\cite{cns1}. We first establish the boundedness of the second tangential derivatives of $u$ on the boundary, followed by the mixed second derivatives. The boundedness of the second normal derivatives is then derived from these estimates and the equation of $u$.

\begin{thm}\label{thm:2.3}
There is a positive constant $C$ such that
\begin{equation}\label{d2u}
  1/C\le \lambda(D^2u)\le C \quad\text{ on }\partial_pQ_T,
\end{equation}
where $\lambda(D^2u)$ denotes the eigenvalues of $D^2u$ and the constant $C$ depends only on on $n, T$, $\Omega$, $c_0$, $||\phi||_{\tilde C^4(\overline{Q}_T)}$, $\|D \psi\|_{L^\infty( {Q}_T)}$, $\|\psi\|_{L^\infty( {Q}_T)}$, $\|\psi_t\|_{L^{n+1}({Q}_T)}$ and the uniform convexity of $\phi(\cdot,0)$.
\end{thm}

\begin{proof}
By the boundary condition, the matrix norm of $D^2u$ is bounded by $D^2\phi(\cdot,0)$ on $\Omega\times\{0\}$. We also note that by the uniform convexity of $\phi$, there exists a positive constant $C$ such that
$$C^{-1}I_n\le D^2u(x,0)\le CI_n$$
for $x\in\Omega$. Here $C$ depends on $||D^2\phi(\cdot,0)||_{L^{\infty}(\Omega)}$ and the uniform convexity of $\phi(\cdot,0)$.
Let $(x_0,t_0)$ be an arbitrary point on  $\partial\Omega\times[0,T]$, w.l.o.g. $x_0=0$ and $\partial\Omega$ be parameterized by $x_n=\rho(x')$ on a small neighborhood of $0$ with $\rho(0)=0$, $D\rho(0)=0$ and $\Omega \subset \{x_n > 0\}$. 
We claim that
\begin{equation}\label{es-x'x'}
C^{-1}I_{n-1}\le D_{x'x'}u(0,t_0)\le CI_{n-1}\quad\text{ for any } t_0\in [0,T].
\end{equation}
Differentiate the identity $u(x',\rho(x'),t)=\phi(x',\rho(x'),t)$ to obtain
\begin{equation}\label{bdy-1}
  (\nabla_{x'}+\nabla\rho\partial_n)(u-\phi)=0.
\end{equation}
Twice differentiation at point $(0,t_0)$ gives that
$$D_{x'x'}u(0,t_0)=-u_n(0,t_0)D^2\rho(0)+D_{x'x'} \phi(0,t_0)+D^2\rho(0)\phi_n(0,t_0).$$
We can obtain the latter inequality in~\eqref{es-x'x'} from this formula, as $|Du|,|D\phi|$ and $|D^2\phi|$ are bounded.
The former inequality of~\eqref{es-x'x'} follows the same proof as in~\cite{cns1} if we regard $u_t+ \psi$ as a data function at the time $t_0$, and the estimation depends on the maximum and minimum values of $u_t+\psi$, $\Omega$ and $\|\phi\|_{\tilde C^4{(\overline{Q}_T)}}$.

Then we shall prove a uniform bound for the mixed second derivatives of $u$.
Suppose that
\begin{equation}\label{boundary}
\rho(x')=\frac{1}{2}\sum_{\alpha,\beta<n} B_{\alpha\beta}x_{\alpha}x_{\beta}+O(|x'|^3)
\end{equation}
in a neighborhood of $0$. Note that the matrix $B_{\alpha\beta}$ is positive definite because of the uniform convexity of $\partial\Omega$.
Denote the operator
\begin{equation}\label{t-alpha}
\mathcal{T}=\partial_{\alpha}+\sum_{\beta<n}B_{\alpha\beta}( x_{\beta}\partial_{n}-x_n\partial_{\beta}).
\end{equation}
Differentiating equation~\eqref{eq}, we obtain
$$\mathcal{L}\mathcal{T}u= \mathcal T \psi$$
for the linear operator $\mathcal{L}$ defined in \eqref{eq:l}.
Thus, by~\eqref{es-t}, we have
\begin{equation}
\begin{split}
    |\mathcal{L}\mathcal{T}(u-\phi)| &\le |\mathcal T \psi| + |\mathcal{L}\mathcal{T}\phi| \\ 
    & \le C \|D_x\psi\|_{L^\infty(Q_T)} + 
 |(u_t+\psi)u^{ij}\partial_{ij}\mathcal{T}\phi- \partial_t\mathcal{T}\phi| \\
  & \le \underline c\Big(1+\sum_{i=1}^n u^{ii}\Big)
\end{split} 
\end{equation}
for some positive constant $\underline c$ depending on $\Omega$, $\|D_x\psi\|_{L^\infty(Q_T)}$,  $||\phi||_{\tilde{C}^3(\overline{Q}_T)}$ and the positive upper bound for $u_t+\psi$.
Moreover, on $\Omega\times\{0\}$, $\mathcal{T}(u-\phi)=0$
and
  \begin{equation}
  |\mathcal{T}(u-\phi)|\le C |x'|^2|u_n-\phi_n| + \Big| \sum_{\beta<n}B_{\alpha\beta} x_n(u_{\beta}-\phi_{\beta})\Big|\le \overline{c}x_n,\quad\text{on }\partial\Omega\times[0,T],
\end{equation}
by the uniform convexity of $\Omega$ and the boundedness of $|Du|, |D\phi|$.

Consider the following barrier function
$$w(x,t)=-a|x|^2+bx_n.$$
Then $\mathcal{L}w=-2a(u_t+\psi)\sum_{i=1}^n u^{ii}$. Note that
$$\sum_{i=1}^n u^{ii} \ge n\Big(\prod_{i=1}^n u^{ii} \Big)^{1/n}\ge n(\det D^2u)^{-1/n}\ge nC^{-1/n},$$
where $C$ depends on the maximum of $u_t+\psi$.
If we take $a$ large enough, then
$$\mathcal{L}(w\pm \mathcal{T}(u-\phi))\le 0.$$
Since $\Omega$ is uniformly convex and bounded, there exists some universal constant $c$ such that $x_n\ge c|x|^2$ on $\overline{\Omega}$. Thus we could take $b\ge a/c+\overline{c}$ to derive
$$w\pm \mathcal{T}(u-\phi)\ge 0 \text{ on }\partial_pQ_T.$$
By using the comparison principle for $u$ and $w$, we obtain that
$$-w\le \mathcal{T}(u-\phi)\le w\quad \text{ in }\overline{Q}_T.$$
Since $w(0,t)=\mathcal{T}(u-\phi)(0,t)=0$, this implies that
$$|\partial_n\mathcal{T}(u-\phi)(0,t)|=|u_{\alpha n}(0,t)-\phi_{\alpha n}(0,t)|\le b,$$
which yields that
\begin{equation}\label{es-an}
  |u_{\alpha n}(0,t)|\le b+||\phi||_{C^2(\overline{Q}_T)}\le C.
\end{equation}

Finally, we turn to the second normal derivatives of $u$. From the equation~\eqref{eq}, we know that
$$u_t(0,t) + \psi(0,t)=\det D_{x'x'}u(0,t)u_{nn}(0,t)+ P\big(D_{x'x'}u(0,t),D_{x'x_n}u(0,t)\big),$$
where $P$ is a $n$-degree polynomial on $D_{x'x'}u(0,t)$ and $D_{x'x_n}u(0,t)$. Combining formula \eqref{es-t}, \eqref{es-x'x'} and \eqref{es-an}, we obtain a uniform bound for $u_{nn}$ on $\partial\Omega\times[0,T]$.

To conclude, we have proved a uniform bound for the second derivatives of $u$ on $\partial_pQ_T$, i.e.
\begin{equation}\label{es-d2bdy}
  |D^2u|\le C\quad \text{   on }~~\partial_pQ_T.
\end{equation}
Since $u_t+\psi$ has a  positive lower bound, then so does the smallest eigenvalue of $D^2u$ on the parabolic boundary according to \eqref{es-d2bdy}. 
Thus estimate \eqref{d2u} is proved.
\end{proof}

\subsection{Step 2: global $C^2$ estimates}
\hfill

Let $U$ be the Legendre transform of $u$ defined on $Q_T^*=\{(y,t):\, y\in Du(\Omega,t),t\in(0,T]\}$, i.e.,
$$U(y,t)=\sup_{x\in \Omega}(y\cdot x-u(x,t)).$$
Then $U$ satisfies
\begin{equation}\label{eq-v}
  -U_t - \frac1{\det D^2U} = -\psi(DU,t) \quad\text{ in }Q_T^*.
\end{equation}
Note that $Du(\Omega,t)$ is not necessarily convex. 
Here $Q_T^*$ is bounded and simply connected due to the smoothness of $u$ and formula~\eqref{es-01}. By estimates~\eqref{es-t} and~\eqref{d2u} above, we know that
\begin{equation}\label{uv}
C^{-1}\le \det D^2U\le C \quad \text{and}\quad ||D^2U||_{L^{\infty}(\partial_pQ_T^*)}\le C.
\end{equation}

\begin{lem}\label{asl}
If $U\in \tilde C^4(\overline{Q_T^*})$ satisfies equation~\eqref{eq-v},
it holds that $$\sup_{Q_T^*}||D^2U||\le  \sup_{\partial_pQ_T^*}||D^2U||.$$
\end{lem}

\begin{proof}
Differentiating the equation~\eqref{eq-v}, we have for $k,l=1,\cdots,n$,
$$ - U_{tk} + \frac1{\det D^2U} U^{ij}U_{ijk}= -\psi_i U_{ik}$$
and 
\begin{equation*}
\begin{split}
     - U_{tkl} + \frac1{\det D^2U} U^{ij}U_{ijkl} -\frac{\big(U^{ij}U_{ijk} \big) \big(U^{ij}U_{ijl} \big)}{\det D^2U}  - \frac{U_{ijk}U_{pql}U^{ip}U^{qj}}{\det D^2U}   
    =  -\psi_{ij}U_{ik}U_{jl} - \psi_i U_{ikl}.
\end{split}
\end{equation*}
If we define the linear operator $\tilde{L}= - \partial_t+ \frac1{\det D^2U} U^{ij}\partial_{ij} + \psi_i \partial_i$, then 
\begin{align*}
  \tilde{L}U_{kk}&=\frac1{\det D^2U} \Big(U^{ij}U_{ijk} \Big)^2 + \frac1{\det D^2U} U_{ijk}U_{pqk}U^{ip}U^{qj} -\psi_{ij}U_{ik}U_{jk}.
\end{align*}
Since $U$ is convex and $\psi$ is concave, it holds that $U_{ijk}U_{pqk}U^{ip}U^{qj} \ge 0$ and $-\psi_{ij}U_{ik}U_{jk} \ge 0$. 
Therefore,

\begin{equation}\label{eq:tl-ukk}
    \tilde{L}U_{kk}\ge 0.
\end{equation}
By the comparison principle (\cite[Theorem 2.4]{lie}) for linear parabolic equations, we obtain that
$$U_{kk}\le  \sup_{\partial_pQ_T^*}U_{kk},$$
or,
$$\sup_{Q_T^*}||D^2U||\le \sup_{\partial_pQ_T^*}||D^2U||.$$
\end{proof}

By formula~\eqref{uv} and Lemma~\ref{asl}, the eigenvalues of $D^2u$ have a positive lower bound in $Q_T$. Combined with~\eqref{es-t}, they are also bounded from above. To be more explicit, we have proved the following theorem:

\begin{thm}\label{c2thm}

It holds that
\begin{equation}\label{eq:d2u-c2}
 1/C\le \lambda(D^2u)\le C  \quad\text{ on }~~ \overline{Q}_T,
\end{equation}
where $C$ is a positive constant depending on those universal quantities in Theorem \ref{thm:2.3}.
\end{thm}

\subsection{Step 3: global $\tilde{C}^{2,\alpha}$ estimates}\label{1.3}
\hfill

In this part, we first give the interior $\tilde{C}^{2,\alpha}$ regularity for $u$.

\begin{thm}\label{hi}
Recall that $u\in \tilde{C}^{4}(\overline{Q}_T)$ is a solution to equation~\eqref{eq} and is convex in $x$.
Then, for any $Q\subset\subset Q_T$, there exists $\alpha \in (0,1)$ such that
  $$||u||_{\tilde{C}^{2,\alpha}(\overline{Q})}\le C,$$
where $\alpha, C$ are positive constants depending on those universal quantities in Theorem \ref{thm:2.3} and $\text{dist\,}(Q,Q_T)$.
\end{thm}

\begin{proof}
    Recall that $U_{kk}$ satisfies \eqref{eq:tl-ukk} as $\psi(x,t)$ is concave with respect to $x$.
By using the same method in \cite[Theorem 4.1]{tso}, it holds that $U\in \tilde C^{2,\alpha}(\Omega \times(0,T])$.
Then by \eqref{eq:d2u-c2}, we have the interior regularity for $u$ in $Q_T$.
\end{proof}

Recall that 
\begin{equation*}
  \mathcal L u_t =\psi_t \quad\text{ in } \ Q_T,  
\end{equation*}
where the linear operator $\mathcal L = - \partial_t + \det D^2 u \cdot u^{ij} \partial_{ij} $ is defined in \eqref{eq:l}.
By estimate \eqref{eq:d2u-c2}, the operator $\mathcal L$ is uniformly parabolic.
Thanks to the Krylov-Safonov's estimate for linear parabolic equations (see \cite[Theorem 5.1]{gr} or \cite{KS}), we can obtain the global $\tilde{C}^{\alpha}$ estimate to $u_t$.

\begin{thm}\label{ut}
Under the same conditions as in Theorem~\ref{hi}, we have
$$||u_t||_{\tilde{C}^{\alpha}(\overline{Q}_T)}\le C$$
for some positive constants $C$ and $\alpha$ depending on the uniformly parabolic coefficients of $\mathcal L$ and $\|\psi_t\|_{L^{n+1}(Q_T)}$.
\end{thm}

\begin{rem}
The interior $\tilde{C}^{2,\alpha}$ regularity of $u$ can also be obtained by using the same method in \cite[Proposition 5.3]{CDKL}, once we get the global $C^2$ estimates and the $\tilde{C}^{\alpha}$ estimate for $u_t$.
\end{rem}

We split the proof of the global $\tilde{C}^{2,\alpha}$ estimate into two parts. A short-time existence theorem gives the global $\tilde{C}^{\alpha}$ estimate for $D^2u$ on $\Omega\times(0,2\epsilon]$.
As for the estimate on the rest part of $Q_T$, considering the linear parabolic equation satisfied by tangential derivatives of $u$, we use the approach of Theorem 9.31  in~\cite{gt} to deduce the H\"older regularity of the gradient, which is independent of the moduli of continuity of the principal coefficients.

Let $Y_0=(y_0,t_0)$ be any point on $\partial\Omega\times[\epsilon,T]$ with some constant $\epsilon>0$. Without loss of generality, assume that $y_0=0$, and on a small neighborhood of $0$, $\partial\Omega$ is the image of
$y_n=\rho(y')$ with $\rho(0)=0$ and $D\rho(0)=0$.
Perform the transformation of coordinates $x=\Psi(y)$, and $\Psi:B_{\rho_0}(0)\cap \overline{\Omega}\to \overline{\mathbb{R}_n^+}$, with
$$x'=y',\quad x_n=y_n-\rho(y').$$
Let $$v(x,t)=T_k(u-\phi)(x',x_n+\rho(x'),t)$$
for a fixed $1\le k<n$, where
$$T_i:=\partial_{y_i}+\rho_{y_i}\partial_{y_n}\quad \text{ for }~i<n,\quad T_n=\partial_{y_n}.$$
Take $G=B_{R_0}^+\times(t_0-R_0^2,t_0]$, such that $B_{R_0}^+\subset \operatorname{Im}\Psi$ and $R_0^2<\epsilon$. 
Then by direct computations,
\begin{equation*}
  \begin{aligned}
    \mathcal L (T_ku) & = -\partial_t(T_ku)+ \det D^2 u \cdot u^{ij}(T_ku)_{ij} \\ 
    & = T_k \psi + \det D^2 u \cdot (u^{ij}\rho_{kij}u_n+ u^{ij}\rho_{ki}u_{nj}+ u^{ij}\rho_{kj}u_{ni}) =:\overline{f}.
  \end{aligned}
\end{equation*}
Denote $P_i=\partial_{x_i}-\rho_{x_i}\partial_{x_n}$ for $1\le i< n$ and $P_n=\partial_{x_n}$.
Thus under $(x,t)$ coordinates,
$$-\partial_t v+\det D^2 u \cdot u^{ij}P_iP_jv=Lv-\det D^2 u \cdot u^{ij}\rho_{ij}v_n.$$
Here the linear operator $L=-\partial_t+a^{ij}\partial_{x_ix_j}$ with
$$a^{ij}=\det D^2 u \cdot (u^{ij}+ u^{sl}\rho_s\rho_l\delta_n^i \delta_n^j- u^{sj}\rho_s\delta_n^i- u^{is} \rho_s\delta_n^j), $$
where $\delta_n^i$ equals $1$ when $n=i$, otherwise $0$, and the summation over $s,l$ is from $1$ to $n-1$. 
Note that $P_jv =\partial_{y_j}(T_k(u-\phi))$ and $P_iP_jv =\partial_{y_iy_j}(T_k(u-\phi))$ for $i,j=1,\cdots,n$.

Let
$$f(x,t) = \Big(\det D^2 u \cdot u^{ij}\rho_{ij} \partial_{y_n}(T_k(u-\phi))- \mathcal L(T_k\phi)+\overline{f} \Big)(\Psi^{-1}(x),t).$$
Then, $v$ satisfies the equation
\begin{equation}\label{v-eq}
  \left\{\begin{aligned}
    -v_t+a^{ij}v_{ij}&=f\quad\text{   in }G,\\
    v&=0\quad\text{   on }F,
  \end{aligned}\right.
\end{equation}
where $F=\partial G\cap\{x_n=0\}$ is the flat boundary.

By Lemma~\ref{lem1} and Theorem~\ref{c2thm}, we have
\begin{equation}\label{v-ell}
\lambda I_n\le (a^{ij})\le \Lambda I_n,
\end{equation}
if $\rho_0>0$ small, and
\begin{equation}\label{v-ell-b}
|f|\le C_0,\quad\quad |v|+|Dv|\le C_1 \quad \text{in}\quad \overline{G}
\end{equation}
for some positive constants $\lambda$, $\Lambda$, $C_0$ and $C_1$ depending on those universal quantities in Theorem \ref{thm:2.3}.

Before proving the boundary $\tilde{C}^{1,\alpha}$ regularity of $v$, we first give three lemmas including the weak Harnack inequality, the interior $\tilde C^{1,\alpha}$ regularity of $v$, and a technical lemma.

\begin{lem}\label{v-osc}
Let $v\in \tilde{C}^2(\overline{G})$ be a solution of equation~\eqref{v-eq} with conditions~\eqref{v-ell} and~\eqref{v-ell-b}.
Then there are constants $0<\alpha<1$ and $C>0$ depending only on $n,\lambda,\Lambda$, $\operatorname{diam} G$, $C_0$ and $C_1$, such that for any $Q(X_1,R)\subset G$, $r\le R$,
$$\osc_{Q(X_1,r)}Dv\le C\Big(\frac{r}{R}\Big)^{\alpha}(\osc _{Q(X_1,R)}Dv+R),$$
where $Q(X_1,R)=\{(x,t):|x-x_1|<R,t_1-R^2<t\le t_1\}$ with $X_1=(x_1,t_1)$.
\end{lem}

This lemma is derived from Theorem 12.3 in~\cite{lie}. Note that this lemma  reveals the interior $\tilde{C}^{1,\alpha}$ estimate of $v$. The next lemma is the weak Harnack inequality for linear parabolic equations, borrowed from Lemma 7.37 in~\cite{lie}.

\begin{lem}\label{v-whn}
Under the same conditions as in Lemma~\ref{v-osc}. Suppose also that $v$ is nonnegative. Then for any $Q(X_1,4R)\subset G$,  we have
  $$\Big(R^{-n-2}\int_{\Theta(X_1,R)}v^p\,dX\Big)^{1/p}\le C\Big(\inf_{Q(X_1,R)}v+R^{\frac{n}{n+1}}||f||_{L^{n+1}(G)}\Big).$$
  Here $\Theta(X_1,R)=Q((x_1,t_1-4R^2),R)$ and $p, C$ are positive constants depending only on $n,\lambda,\Lambda$.
\end{lem}

Additionally, the next lemma is Lemma 12.4 in~\cite{lie}.

\begin{lem}\label{v-bri}
Let $w\in C(\overline{Q})$ and $Dw\in C(Q)$ for some cylinder $Q=Q(X_0,R)$. Suppose that there are nonnegative constants $\bar C_1, \bar C_2$ and $\alpha<1$ such that
  $$\osc\limits_{Q(X_1,r)}Dw\le \bar C_1 \Big(\frac{r}{\mu}\Big)^{\alpha}(\osc\limits_{Q(X_1,\mu)}Dw +\bar C_2 \mu^{\alpha}),$$
  whenever $Q(X_1,\mu)\subset Q$ and $r\le \mu$. Then for any $\ell\in \mathbb{R}^n$ and $a\in \mathbb{R}$, we have
  $$\sup_{Q(X_0, R/2)}|Dw-\ell| \le C(\bar C_1, n , \alpha)\Big(R^{-1}\sup_Q|w(X)-\ell \cdot x-a|+\bar C_2 R^\alpha\Big).$$
\end{lem}

We will use Lemma~\ref{v-whn} to prove an oscillation estimate for $\frac{v}{x_n}$, which gives the H\"older estimate for $v_n$ on $F$, the flat boundary.
Combining this with the interior estimates, Lemma~\ref{v-osc}, and with the help of Lemma~\ref{v-bri}, we will derive a global $\tilde{C}^{1,\alpha}$ estimate for $v$ in $\frac12 G := B_{R_0/2}^+\times(t_0-R_0^2/4,t_0]$.
Recall the definition of $v$, then the global H\"older estimate of $D^2u$ on $\Omega\times(\epsilon,T]$ follows from equation \eqref{eq}.

\begin{thm}\label{t2.11}
Let $v\in \tilde{C}^2(\overline{G})$ be a solution of equation~\eqref{v-eq} with conditions~\eqref{v-ell} and~\eqref{v-ell-b}.
Then there are small positive constants $r_0$ and $\delta$ depending on $n,\lambda,\Lambda, R_0$ and $C_0$ such that for all $r\in(0, r_0)$ and
$X_1\in Q_{r,\delta}=\{(x',x_n,t):|x'|<r,0<x_n<\delta r, t_0-r^2<t\le t_0\},$
we have
  $$|Dv(X_1)-Dv(0,t_0)|\le Cr^{\alpha},$$
for some positive constants $C$ and $\alpha<1$ depending on $n, C_0,C_1,\lambda,\Lambda,r_0,\delta$ and $R_0$.
\end{thm}

\begin{proof}
Take $r<r_0$ small such that $r_0+r_0/\delta\le 1/2R_0$. $\delta<1$ is a small constant to be chosen later.
For any $X_1=(x_1,t_1)\in Q_{r,\delta}$, let $\ell=Dv(X_1')$, where $X_1'=(x_1',0,t_1)=(x_{1,1},\cdots,x_{1,n-1},0,t_1)\in F$.
Since $v=0$ on $F$, $\ell=(0,\cdots,0,\ell_n)$. 
Take any point $$X=(x,t)\in Q(X_1,\mu)=\{|x-x_1|<\mu,t_1-\mu^2<t\le t_1\}\quad \text{with $\mu=\frac12x_{1,n}$}.$$
If we take $\sigma = r+\mu/\delta$, then $\sigma<2r$ and
$X,X_1'\in Q_{\sigma,\delta} \subset Q_{R_0/2,\delta}$.
Consider the function $v(X)-\ell\cdot x$, then by Lemma~\ref{v-osc} and Lemma~\ref{v-bri},
\begin{equation}\label{it}
\begin{split}
    |Dv(X_1)-Dv(X_1')|&\le\sup_{Q(X_1,\mu/2)}|Dv-\ell| \\
    &\le C\Big(\mu^{-1}\sup_{Q(X_1,\mu)}|v(X)-\ell\cdot x|+\mu^{\alpha}\Big).
\end{split}
\end{equation}

We claim that
\begin{equation}\label{claim}
\osc\limits_{Q_{\sigma,\delta}}\frac{v(X)}{x_n}\le C\sigma^{\alpha}.
\end{equation}
  Then since $v|_F=0$ and $X,X_1'\in Q_{\sigma,\delta}$, we get
  $$\Big|\frac{v(X)}{x_n}-\ell_n\Big|\le \osc\limits_{Q_{\sigma,\delta}}\frac{v(X)}{x_n}\le C\sigma^{\alpha}.$$
  Thus~\eqref{it} gives
  \begin{equation*}
    \begin{aligned}
      |Dv(X_1)-Dv(X_1')|&\le C\Big(\mu^{-1}\sup_{Q(X_1,\mu)}Cx_n\sigma^{\alpha}+ \mu^{\alpha}\Big)\\
      &\le C(\mu^{-1}\sigma^{\alpha}(x_{1,n}+\mu)+\mu^{\alpha})\\
      &\le C(\sigma^{\alpha}+\mu^{\alpha})\\
      &\le Cr^{\alpha},
    \end{aligned}
  \end{equation*}
  where we use the fact that $\mu=\frac12x_{1,n} < \delta r$ and $\sigma < 2r$.
  Again by the claim, 
  $$|Dv(X_1')-Dv(0,t_0)|\le C\sigma^{\alpha}\le Cr^{\alpha},$$ 
  then we have proved the theorem by the triangle inequality. For the rest part of this proof, we will show the claim~\eqref{claim}.

First suppose that $v\ge 0$ and denote 
$$F_{R,\delta}=:\{(x,t): |x'|<R,x_n=\delta R, t_0-R^2<t\le t_0\}.$$
We will prove the following inequality
  \begin{equation}\label{in}
    \inf_{F_{R,\delta}}\frac{v}{x_n}\le 10\Big(\inf_{Q_{R/2,\delta}}\frac{v}{x_n} +R \sup |f|\Big)\quad \text{ for $R<1/4R_0$.}
  \end{equation}
 Without loss of generality, assume that
 \begin{equation}\label{as}
  R=1,\;\inf_{F_{1,\delta}}\frac{v}{x_n}=1.
 \end{equation}
 Let $$w(x,t)=\left[1-|x'|^2+(1+\sup|f|)\frac{x_n- \delta}{\sqrt{\delta}}+t-t_0\right]x_n,$$
then for $L=-\partial_t+a^{ij}\partial_{ij}$, we have
$$Lw\ge- \delta+\frac{2(1+\sup|f|)}{\sqrt{\delta}}\lambda -c(n) \Lambda(1+\delta) \quad\text{in $Q_{1,\delta}$.}$$
Hence, $Lw\ge \sup |f|\ge Lv$ in $Q_{1,\delta}$, when $\delta= \delta(\lambda, \Lambda, c(n))>0$ is small enough.
On $\{x_n=0\} \cap \partial_pQ_{1,\delta}$, $w=0=v$. On $\{x_n=\delta\}\cap \partial_pQ_{1,\delta}$, $w\le \delta(1-|x'|^2)\le\delta \le v$ by~\eqref{as}. 
Additionally, on $\{|x'|=1\}\cap \partial_pQ_{1,\delta}$ and $\{t=t_0-1\}\cap \partial_pQ_{1,\delta}$, $w<0 \le v$. Thus by maximal principle, we have
$w\le v$ in $Q_{1,\delta}$.
Therefore, in $Q_{1/2,\delta}$,
\begin{equation*}
  \begin{aligned}
 \frac{v}{x_n}&\ge 1-|x'|^2+(1+\sup |f|)\frac{x_n\delta}{\sqrt{\delta}}+t-t_0\\
    &\ge 1-\frac{1}{4}-\sqrt{\delta}(1+\sup|f|)-\frac{1}{4}\\
    &\ge \frac{1}{10}-\sup|f|.
  \end{aligned}
\end{equation*}
when $\delta<1/100$.
Now the constant $\delta$ is determined, so does $r_0$.

Next, without the assumption~\eqref{as}, let $$m_1=\inf_{
F_{R,\delta}}\frac{v}{x_n},\quad\text{and}\quad \tilde{v}(x,t)=\frac{v(Rx,R^2(t-t_0)+t_0)}{m_1R}.$$
Then $-\tilde{v}_t+ a^{ij}\tilde{v}_{ij}=\tilde{f}$ in $Q_{1,\delta}$, with 
$\tilde{f}=\frac{R}{m_1}f$.
Note that $\inf_{F_{1,\delta}}\frac{\tilde{v}}{x_n}=1$. Then from the above arguments,
$$\frac{\tilde{v}(x,t)}{x_n}\ge 1/10-\sup|\tilde{f}|.$$
Hence in the original coordinates, we have
$$\frac{v}{x_n}\ge \frac{m_1}{10}- R \sup|f|\quad \text{in $Q_{R/2,\delta}$,}$$
which yields~\eqref{in}.

Denote
$$Q_{R,\delta}^*=\{|x'|<R,\frac{1}{2}\delta R<x_n<\delta R,t_0-R^2<t\le t_0\},$$
$$Q'_{R,\delta}=\{|x'|<R,\frac{1}{2}\delta R<x_n<\delta R,t_0-3R^2<t\le t_0-2R^2\},$$
and
$$m=\inf_{Q_{2R,\delta}}\frac{v}{x_n}, \quad\quad M=\sup_{Q_{2R,\delta}}\frac{v}{x_n}.$$
Then $Mx_n-v\ge 0$, $v-mx_n\ge 0$ in $Q_{2R,\delta}$.
Applying Lemma~\ref{v-whn}, the weak Harnack inequality, to $Mx_n-v$ and $v-mx_n$ in $Q^*_{R,\delta}$, we have
\begin{equation}\label{m-sm}
  \begin{split}
    \Big(R^{-n-2}\int_{Q'_{R,\delta}}(Mx_n-v)^p\Big)^{1/p} &\le C\Big(\inf_{Q^*_{R,\delta}}(Mx_n-v)+R^2\sup |f|\Big)\\
    &\le C(R\inf_{F_{R,\delta}} \Big(M-\frac{v}{x_n})+R^2\sup|f|\Big)\\
    &\le C\Big(R\inf_{Q_{R/2,\delta}}(M-\frac{v}{x_n})+R^2\sup|f|\Big)\\
    &\le CR\Big(M-\sup_{Q_{R/2,\delta}}\frac{v}{x_n}+R\sup|f|\Big).
  \end{split}
\end{equation}
The third inequality above is deduced by~\eqref{in}.
Similarly we have
\begin{equation}\label{m-bm}
\Big(R^{-n-2}\int_{Q'_{R,\delta}}(v-mx_n)^p\Big)^{1/p}\le CR\Big(\inf_{Q_{R/2,\delta}}\frac{v}{x_n}-m+R\sup|f|\Big).
\end{equation}
Add equations~\eqref{m-sm} and~\eqref{m-bm} up, and by Minkowski inequality,
$$(M-m)\Big(R^{-n-2}\int_{Q'_{R,\delta}}x_n^p\Big)^{1/p}\le CR\Big(M-m-\osc\limits_{Q_{R/2,\delta}}\frac{v}{x_n}+ R\sup|f|\Big).$$
Direct computations give that
$$\osc\limits_{Q_{R/2,\delta}}\frac{v}{x_n}\le C\Big(\osc\limits_{Q_{2R,\delta}}\frac{v}{x_n}+R\sup|f|\Big).$$
Then it follows from the classical iterative arguments (see Lemma 8.23 in~\cite{gt} for a reference) that
\begin{equation}\label{ys}
  \osc\limits_{Q_{R,\delta}}\frac{v}{x_n}\le C\Big(\frac{R}{R_0}\Big)^{\beta}\Big(\osc\limits_{Q_{R_0, \delta}}\frac{v}{x_n}+R_0\sup|f|\Big)
\end{equation}
for any $R<R_0$ and some $\beta<1$ depending on $n,C_0,C_1,\lambda,\Lambda,\delta$. We still use $\alpha$ to denote $\min\{\alpha,\beta\}$. Since $|v|+|Dv|$ and $|f|$ are bounded in $\overline{G}$, the claim~\eqref{claim} is now proved.
\end{proof}

Since $v|_F=0$, the inequality~\eqref{ys} implies the H\"older continuity of $Dv$ on $F$. With the above theorem, we infer that,
$$
  |Dv(X)-Dv(Y)|\le C(|x-y|+|t-s|^{1/2})^{\alpha}
$$
for any $X\in G$ and $Y\in F$. The $\tilde{C}^{\alpha}$ estimate for $Dv$ on $\overline{G}$ then follows from the triangle inequality and Lemma~\ref{v-osc}.

Reverting to $u$, we have

\begin{cor}\label{up}

For any $\epsilon>0$, there exist positive constants $C$ and $\alpha$ such that
$$||D^2u||_{\tilde{C}^{\alpha}(\overline{\Omega}\times(\epsilon,T])}\le C,$$
where the constants $C$ and $\alpha$ depend only on $\epsilon^{-1}$ and those universal quantities in Theorem \ref{thm:2.3}.
\end{cor}

\begin{proof}
Recall that $v(x,t)=T_k(u-\phi)(x',x_n+\rho(x'),t)$ in $\Psi(B_{\rho_0}(Y_0))$ with $T_k=\partial_{y_k}+\rho_{y_k}\partial_{y_n}$, $1\le k<n$, where we take $\rho_0>0$ small enough.
If we define $\tilde{u}(x,t)=u(x',x_n+\rho(x'),t)$, then $\tilde{u}_{ij}=T_iT_ju$, for any $1\le i,j\le n$. What we have proved above is equivalent to
\begin{equation}\label{rai}
  ||\tilde{u}_{kl}||_{\tilde{C}^{\alpha}(\overline{G})}\le C\text{   for any }1\le k< n\text{ and }1\le l\le n.
\end{equation}
To prove the $\tilde{C}^{\alpha}$ estimate for $\tilde{u}_{nn}$, note that the equation~\eqref{eq} can be written as $\det(P_iP_j\tilde{u})=\tilde{u}_t + \psi(\Psi^{-1}(x),t)$, with
$P_k=\partial_{x_k}-\rho_{x_k}\partial_{x_n}$ for $1\le k< n$ and $P_n=\partial_{x_n}$,
which is equivalently
$$Q_1\tilde{u}_{nn}+Q_2=\tilde{u}_t + \psi(\Psi^{-1}(x),t),$$
where $Q_1$ and $Q_2$ are polynomials of $\tilde{u}_{kl}$ with $k+l<2n$ and $\rho$ up to the second derivatives.
By Theorem~\ref{c2thm}, the eigenvalues of the Hessian $D^2u$ have positive lower and upper bounds; therefore, $Q_1$ has a positive lower bound.
This fact together with Theorem~\ref{ut} and~\eqref{rai} leads to the $\tilde{C}^{\alpha}$ estimate of $\tilde{u}_{nn}$.
From this the $\tilde{C}^{\alpha}$ estimate of $D^2u$ in $\Omega\times (\epsilon,T]$ follows by the arbitrariness of $Y_0$ and the interior higher regularity of $u$ in Theorem~\ref{hi}.
\end{proof}

\section{Existence and uniqueness}\label{s3}

Equation~\eqref{eq} is uniformly parabolic by virtue of Theorem~\ref{c2thm}. Hence with some compatibility conditions on $\phi$, we have the following short-time existence theorem by the inverse function theorem.
The proof is a modification of the argument in \cite[Theorem 2.5.7]{cv}.

\begin{thm}\label{se}
  Let $\Omega$ be a uniformly convex domain in $\mathbb R^n$ with smooth boundary. 
  Assume that $\varphi$ is a smooth and uniformly convex function on $\overline{\Omega}$, $\psi$ is a smooth function on $\overline{\Omega}\times[0,T]$ and $h$ is a smooth function on $\partial\Omega\times[0,T]$.
Suppose also that the equation
  \begin{equation}\label{sh}
    \left\{\begin{aligned}
      &-u_t + \det D^2u = \psi(x,t), \\
      &u|_{t=0}=\varphi,\quad
      u|_{\partial\Omega}=h,
    \end{aligned}
    \right.
  \end{equation}
  satisfies the compatibility condition of order $1$, i.e.,
  $$h(0)=\varphi|_{\partial\Omega}\quad \text{    and    }\quad -h_t(0) + \det D^2\varphi|_{\partial\Omega} = \psi(\cdot,0)|_{\partial\Omega}.$$
  Then for any $0<\alpha\le 1$, there is a small positive constant $\epsilon$ such that the initial boundary value problem~\eqref{sh} has a unique solution $u\in \tilde{C}^{2,\alpha}(\overline{Q}_{2\epsilon})$, which is uniformly convex in $x$ for all $t\in[0,2\epsilon)$.
  Here we have $||u||_{\tilde{C}^{2,\alpha}(\overline{Q}_{2\epsilon})}\le C$ and $Q_{2\epsilon}=\Omega\times(0,2\epsilon]$. The constants $\epsilon$ and $C$ depend on $n$, $\Omega$, the uniform convexity of $\varphi$, $||\psi||_{\tilde{C}^1(\overline{\Omega}\times[0,T])}$, $||\varphi||_{\tilde{C}^3(\overline{\Omega})}$, and $||h||_{\tilde{C}^2(\partial\Omega\times[0,T])}$.
  
\end{thm}

\begin{proof}
  Consider the equation
\begin{equation}\label{w-p}
  \left\{\begin{aligned}
    &-w_t+\Delta w=\Delta\varphi-\det D^2\varphi+\psi,\\
    &w|_{t=0}=\varphi,\quad w|_{\partial\Omega}=h.
  \end{aligned}
  \right.
\end{equation}
According to theorem 5.2 in Chapter 4 of~\cite{ru}, problem~\eqref{w-p} has a unique solution $w\in\tilde{C}^{2,\alpha}(\overline{Q}_{T_0})$ for some positive constant $T_0$.
Moreover, $w(x,t)$ is still uniformly convex in $x$ for any $0\le t\le T_0$ when $T_0$ is small enough. Indeed, by the $\tilde{C}^{2,\alpha}$ regularity of $w$,
$||w(\cdot,t)-\varphi||_{C^2(\overline\Omega)}\le Ct^{\alpha/2}$,
then the uniform convexity of $w(\cdot,t)$ follows from that of $\varphi$ for sufficiently small $t$.

Let $$g=-w_t+\det D^2 w - \psi \in\tilde{C}^{\alpha}(\overline{Q}_{T_0}).$$
Then by equation~\eqref{w-p} and the compatibility condition, $g|_{t=0}=0$. Let $V$ be an open neighborhood of $w$ in $\tilde{C}^{2,\alpha}(\overline{Q}_{T_0})$. Define the map $\Phi:V\to W$ by
\begin{equation}\label{phi-df}
\Phi(u)=(-u_t+\det D^2u - \psi,~~ u|_{t=0},~~ u|_{\partial\Omega\times[0,T_0]}),
\end{equation}
where $W=\tilde{C}^{\alpha}(\overline{Q}_{T_0})\times S$ and
$$S=\{(a,b)\in C^{2,\alpha}(\overline{\Omega})\times\tilde{C}^{2,\alpha}(\partial\Omega\times[0,T_0]):\;a|_{\partial\Omega}=b|_{t=0}, \det D^2a|_{\partial \Omega}=(\psi+b_t)|_{t=0}\}.$$
One can prove that $V$ and $W$ are Banach manifolds
(see theorem 4.4 in~\cite{K}). Note that $\Phi(w)=(g,\varphi,h)$. The differential of $\Phi$ at $w$ is a linear operator denoted by $L$.
$$L\eta=(-\eta_t+(\det D^2w)w^{ij}\eta_{ij} - \psi_t,~~ \eta(0),~~ \eta|_{\partial\Omega\times[0,T_0]}),\quad\text{for any $\eta\in V$.}$$
 The construction of $w$ and $S$ ensures the compatibility condition for $L$. Hence by Theorem 5.2 in Chapter 4 of~\cite{ru}, $L$ is an isomorphism from the tangent space of $V$ at $w$ to that of $W$ at $(g,\varphi,h)$.
 Then by the inverse function theorem~\cite[Section 15]{Deim}, $\Phi$ is an $C^1$ diffeomorphism from $B_{\tilde\rho}(w)$, a small neighborhood of $w$ in $\tilde{C}^{2,\alpha}(\overline{Q}_{T_0})$ to an open neighborhood $U$ of $(g,\varphi,h)$ in $W$. Take $\tilde\rho$ small enough, then the functions in $B_{\tilde\rho}(w)$ is still uniformly convex in $x$.

Let $\epsilon\in(0,T_0)$ be small, $\eta_{\epsilon}\in C^{\infty}([0,1])$ such that $0\le \eta_{\epsilon}\le 1$, $0\le \eta_{\epsilon}'\le C\epsilon^{-1}$, and
$\eta_{\epsilon}=0$, if $0\le t\le 2\epsilon$; $\eta_{\epsilon}=1$, if $4\epsilon\le t\le 1$.
Denote $g_{\epsilon}=g\eta_{\epsilon}$. Then $g_{\epsilon}\in \tilde{C}^{\alpha}(\overline{Q}_{T_0})$.
One can prove that $g_{\epsilon}\to g\text{ in } \tilde{C}^{\alpha}(\overline{Q}_{T_0})$ when $\epsilon\to 0$.
Thus by taking $\epsilon$ small enough, $(g_{\epsilon},\varphi,h)\in U\subset W$, and there is a function $u\in B_{\tilde\rho}(w)\subset \tilde{C}^{2,\alpha}(\overline{Q}_{T_0})$ such that $$\Phi(u)=(g_{\epsilon},\varphi,h).$$
Since $g_{\epsilon}=0$ for $0\le t\le 2 \epsilon$,
then $u\in\tilde{C}^{2,\alpha}(\overline{Q}_{2\epsilon})$ is a solution to equation~\eqref{sh} in $\Omega\times(0,2\epsilon]$ and is uniformly convex in $x$. 
\end{proof}

\begin{rem}\label{rem}
  When the initial-boundary condition of equation~\eqref{sh} satisfies the compatibility condition of order $m=[\frac{l+\alpha}{2}]+1$, there exists a unique solution $u\in \tilde{C}^{l+2,\alpha}(\overline{Q}_{\epsilon})$. The proof is almost the same as above, and here we omit it. One can refer to Chapter 4, Section 5 in \cite{ru} for the general definition of $m$-th order compatibility condition. To be more explicit, the compatibility condition of \textit{m}-th order for our equation means that
  $$u^{(k)}(x)|_{x\in \partial\Omega}=\frac{\partial^k h}{\partial t^k}\Big|_{t=0},$$
  for $k=0,\cdots,m$, where $u^{(k)}$ is recursively defined by
  $$u^{(0)}(x)=\varphi(x),\quad u^{(k+1)}(x)=\frac{\partial^k}{\partial t^k}\Big|_{t=0} \Big(\det D^2u(x,t) - \psi(x,t)\Big).$$
  Here the right term of the second equality can be recursively expressed by derivatives of $\varphi$ and $\psi$.
\end{rem}

The uniqueness of solutions to the equation~\eqref{eq} then follows immediately from Lemma \ref{le:cp}.
\begin{cor}\label{uni}
  Let $\Omega$ be a uniformly convex bounded domain with smooth boundary and $\phi$ be a continuous function in $\overline{Q}_T$. Then equation~\eqref{eq} admits at most one solution in $\tilde{C}^2$ that is convex in $x$.
\end{cor}

\subsection{Proof of main theorems}
We are in a position to prove our main theorems.

{\bf Proof of Theorem~\ref{0-pri}:\quad}
By uniqueness, the a priori solution $u$ in Corollary~\ref{up} coincides with the solution derived from Theorem~\ref{se} on $[0,2\epsilon]$ if $\varphi=\phi|_{t=0}$ and $h=\phi|_{\partial\Omega}$. Thus, for some small positive constant $\epsilon$, $\|u\|_{\tilde{C}^{2,\alpha}(\overline Q_{2\epsilon})}$ is bounded.
Combined with Corollary~\ref{up}, this gives the global $\tilde{C}^{2,\alpha}$ estimate to $u$ on $\overline{Q}_T$.
\qed

{\bf Proof of Theorem~\ref{thm1.1}:\quad}
The existence of $\tilde{C}^{2,\alpha}$ solutions of equation~\eqref{eq} can be derived step by step from the short-time existence under the uniform global $\tilde{C}^{2,\alpha}$ estimate. It is also unique by Corollary~\ref{uni}.

Additionally, the higher order regularity of $u$ can also be obtained by classic regularity theory of linear parabolic equations iteratively if $\partial\Omega$ and $\phi$ are smooth enough, and equation~\eqref{eq} admits such solution uniquely if the equation satisfies higher order compatibility conditions in Remark~\ref{rem}. Specifically, let $k\ge 2$, $0<\beta\le 1$. Assume that $\partial\Omega$ is of class $C^{k+2,\beta}$, $\phi\in\tilde{C}^{k+2,\beta}( \overline{Q}_T)$, and equation~\eqref{eq} satisfies the compatibility condition of order $[\frac{k+\beta}{2}]+1$. Then the equation admits a unique solution $u\in \tilde{C}^{k+2,\beta}(\overline{Q}_T)$ that is convex in $x$.
\qed

\subsection{Necessity of condition $(\mathcal P3)$~} \label{s3.1}

Conditions $({\mathcal P1})$, $({\mathcal P2})$ and $({\mathcal P3})$ were introduced by Ivochkina and Ladyzhenskaia~\cite{I} to study \eqref{eq:1.1}.
Note that conditions $({\mathcal P1})$ and $({\mathcal P2})$ are natural requirements to ensure $C^{-1} \le \det D^2u\le C$ and thus guarantee the convexity of $u$. We will take counterexamples to show that condition $({\mathcal P3})$ is also necessary for the existence of convex solutions to~\eqref{eq}.

In one space dimension, take $\Omega=(0,1)$ and $T>1$. Then equation~\eqref{eq} reduces to the following initial-boundary value problem
\begin{equation*}
  \left\{\begin{aligned}
    &-v_t + v_{xx} = \psi(x,t)  &&\text{ in } Q_T=(0,1)\times(0,T],\\
    &v=\phi &&\text{ on }\partial_pQ_T,
  \end{aligned}\right.
\end{equation*}
where $\psi$ and $\phi$ are smooth functions satisfying conditions $({\mathcal P1})-({\mathcal P3})$. Assume $v$ is a smooth solution to this equation that is convex in $x$. Define a smooth bump function $w:\overline{Q}_T\to \mathbb{R}$ by
\begin{equation}\label{eq:bump}
w(x,t)=\left\{\begin{aligned}
&At\cdot e^{\frac{-B}{x^2(1-x)^2}} &&\text{ in } Q_T,\\
&0 &&\text{ on }\partial_pQ_T,
\end{aligned}\right.
\end{equation}
where $A,B$ are positive constants.
The function $w$ is non-negative, smooth up to the boundary, and satisfies that
\begin{equation}\label{eq:ex-u}
    w_t =0 \ \text{on} \ \partial\Omega \times [0,T]; \quad \frac{\partial^k w}{\partial x^k} =0 \ \text{on} \ \partial_p Q_T \ \ \text{for} \ k =0,1, 2, \cdots.
\end{equation}
Set $u=v+w$; then we have 
\begin{equation}\label{eq:ex-u-1}
  \left\{\begin{aligned}
    &-u_t + u_{xx} = \psi + (-w_t+w_{xx}) =: \psi+\rho  &&\text{ in } Q_T,\\
    &u=\phi &&\text{ on }\partial_pQ_T.
  \end{aligned}\right.
\end{equation}
Since \eqref{eq:ex-u} gives that $\rho=0$ on $\partial\Omega\times[0,T]$, then conditions $({\mathcal P1})$ and $({\mathcal P2})$ are satisfied for $\phi$ and $\psi+\rho$.
Direct computation yields
\begin{equation*}
w_t=Ae^{\frac{-B}{x^2(1-x)^2}},
\quad w_{xx} = -2ABt\cdot e^{\frac{-B}{x^2(1-x)^2}}\frac{P_6(x)}{x^6(x-1)^6}.
\end{equation*}
Here $P_6(x)$ is a sixth-degree polynomial given by 
$$P_6(x)=2B-8Bx+(8B-3)x^2+16x^3-33x^4+30x^5-10x^6.$$
One can check that there exist constants $x_0\in (0,1/2)$ and $c_1,c_2>0$ depending only on $B$ such that 
$$\inf_{x\in(0,1)}\rho(x,1)=\rho(x_0,1)=-A\cdot c_1<0,\quad\sup_{x\in(0,1)}\rho(x,1)=\rho(1/2,1)=A\cdot c_2>0.$$
Then by choosing $A$ sufficiently large, $\rho$ attains both a negative minimum and a positive maximum, each with large absolute values. This destroys the concavity of $\psi+\rho$. 
For any fixed $t$, the maximum of $w(\cdot,t)$ on $(0,1)$ is $At\cdot e^{-16B}$. 
Then, the maximum of $w$ can be made arbitrarily large and $u$ fails to remain convex in $x$; 
therefore, by the comparison principle, equation \eqref{eq:ex-u-1} does not admit any solution that is convex in $x$ for all time.

In $n$ space dimension, a counterexample can also be constructed similarly for rotationally symmetric solutions.
Let $v(r,t)$ be a smooth radial solution on $(0,1)\times(0,T]$ satisfying
\begin{equation*}
  \left\{\begin{aligned}
    &-v_t + (v_r/r)^{n-1}v_{rr} = \psi(r,t)  &&\text{ in } Q_T=(0,1)\times(0,T],\\
    &v=\phi &&\text{ on }\partial_pQ_T.
  \end{aligned}\right.
\end{equation*}
And $\bar{v}(x,t)=v(|x|,t)$ solves the parabolic Monge–Amp\`ere equation
\begin{equation*}
  \left\{\begin{aligned}
    &-\bar{v}_t+\det D^2\bar{v}= \psi(|x|,t) &&\text{ in } U_T=B_1(0)\times(0,T],\\
    &\bar{v}=\phi(|x|,t) &&\text{ on }\partial_pU_T.
  \end{aligned}\right.
\end{equation*}
We use the same one‐dimensional bump $w(r,t)$ as in \eqref{eq:bump} above and set $\bar{w}(x,t)=w(|x|,t)$. 
The equation for $\bar u= \bar v+ \bar w$ becomes
\begin{equation}\label{eq:ex-u-5}
  \left\{\begin{aligned}
    &- \bar u_t + \det D^2 \bar u  = \bar\Psi(x,t)  &&\text{ in } U_T,\\
    &\bar u=\phi(|x|,t) &&\text{ on }\partial_p U_T,
  \end{aligned}\right.
\end{equation}
where
\begin{equation*}
    \bar\Psi(x,t) = \Psi(|x|,t) = \Psi(r,t) =-v_t-w_t+\big(\frac{v_r+w_r}{r}\big)^{n-1}v_{rr}+\big(\frac{v_r+w_r}{r}\big)^{n-1}w_{rr}.
\end{equation*}
By \eqref{eq:ex-u}, $\bar \Psi$ coincides with $\psi$ on $\partial B_1(0)\times[0,T]$, and $\bar u$ matches $\bar v$ on $\partial_p U_T$. 
Thus, conditions $({\mathcal P1})$ and $({\mathcal P2})$ hold for arbitrary $A$ and $B$. 
For sufficiently large $A$, there exist $r_0\in (1/2,1)$ and $c_3,c_4>0$ depending on $B$ such that $w_r(r_0,1)=A\cdot c_3>0$ and $w_{rr}(r_0,1)=- A\cdot c_4<0$. Since $||v||_{\tilde{C}^2(\overline{Q}_T)}$ is bounded, $\Psi(r_0,1)$ can be made arbitrarily negative, which destroys the concavity of $\Psi$ in $r$. Additionally, $u$ loses convexity in $r$ when $w$ attains a sufficiently large maximum. From the comparison principle Lemma~\ref{le:cp}, we know that there is no convex solution to equation \eqref{eq:ex-u-5}.

\section{Gauss curvature flow}\label{s4}

In  this section, we consider the Gauss curvature flow with boundary
\begin{equation}\label{gauss}
  \left\{\begin{aligned}
    u_t&=\frac{\det D^2u}{(1+|Du|^2)^{\frac{n+1}{2}}}\quad\ \text{ in }Q_T,\\
    u&=\phi\quad\quad\quad\quad\quad\quad\quad\,\;\, \text{on }\partial_pQ_T.
  \end{aligned}\right.
\end{equation}
where $\phi$ and $Q_T$ are defined as before. We will obtain the {\it a priori} estimate for equation~\eqref{gauss} using the same method as before.

\begin{thm}\label{gau}
  Let $u\in \tilde{C}^4(\overline{Q}_T)$ be convex in $x$ and solve equation~\eqref{gauss}.
Suppose that $\Omega$ is a uniformly convex domain in $\mathbb{R}^n$, $\partial\Omega\in C^4$, $\phi\in \tilde{C}^4(\overline{Q}_T)$, $\phi_t\ge c_0 >0$ on $\partial\Omega\times[0,T]$ and $\phi(\cdot,0)$ is uniformly convex. Then we have
$$||u||_{\tilde{C}^{2,\alpha}(\overline{Q}_T)}\le C$$
for some positive constants $C$ and $\alpha$ depending only on $n, T, \Omega,c_0,||\phi||_{\tilde{C}^4(\overline{Q}_T)}$ and the uniform convexity of $\phi(\cdot,0)$.
\end{thm}

\begin{proof}

The proof of this theorem is similar to that of Theorem~\ref{0-pri}, except that the linearized equation of~\eqref{gauss} has lower order terms concerning $Du$.

\textit{(a) $C^0$ and $C^1$ estimates}. Note that $u$ reaches its maximum and minimum value on $\partial_pQ_T$, then $||u||_{C^0(\overline{Q}_T)}\le ||\phi||_{C^0(\overline{Q}_T)}\le C$.

For convenience we denote $\beta=\frac{n+1}{2}$. Differentiating the equation~\eqref{gauss}, we have
$$\mathcal{L}'u_k=0,\quad\text{for}\ \ k=1,\cdots,n,\quad \text{and}\quad \mathcal{L}'u_t=0$$
with
\begin{equation}
\begin{aligned}
    \mathcal{L}'&=-\partial_t+\frac{\det D^2u}{(1+|Du|^2)^{\beta}}u^{ij}\partial_{ij}-2\beta\frac{\det D^2u}{(1+|Du|^2)^{\beta+1}}u_i\partial_i\\
    &=-\partial_t+u_tu^{ij}\partial_{ij}- \frac{2\beta }{1+|Du|^2}u_tu_i\partial_i ~.
\end{aligned}
\end{equation}

Then $\mathcal{L}'$ is a linear parabolic operator. It follows that $|u_t|$ is bounded by $\sup_{\partial_pQ_T}|u_t|$ and
$$C^{-1}\le u_t\le C$$
for some positive constant $C$ depending on $n, c_0,\Omega$, $||\phi||_{\tilde{C}^2}$ and the uniform convexity of $\phi(\cdot,0)$.

To prove the bound of $|Du|$, we apply the similar auxiliary function in~\cite{urb}. Fix any $t_0\in [0,T]$. Let $\tilde{\phi}$ be the solution of the equation
 $$ \left\{\begin{aligned}
    \det D^2\tilde{\phi}&=1\quad\quad\quad\quad\text{in }\Omega,\\
    \tilde{\phi}&=\phi(\cdot,t_0)\quad \, \text{  on }\partial\Omega.
  \end{aligned}\right.$$
The existence and $C^{2}$ bound of solutions to this equation are guaranteed by~\cite{cns1}.
We consider the barrier function
\begin{equation}\label{aux-func}
    \overline{w}=\tilde{\phi}-\frac{1}{q}\log(1+kd_{\partial\Omega})
\end{equation}
defined in $\Omega_{\delta}=\{x\in\Omega:\,d_{\partial\Omega}(x)<\delta\}$ with small positive constant $\delta$ such that $\overline{w} \in C^4(\Omega_{\delta})$.

The positive constants $\delta, q$, and $k$ shall be determined later.
Direct computations give that
$$\begin{aligned}
\frac{\det D^2\overline{w}}{(1+|D\overline{w}|^2)^{\beta}}&\ge \frac{1+Ck^{n+1}q^{-n}(1+kd_{\partial\Omega})^{-n-1}}{(1+ 2|D\tilde{\phi}|^2+2k^2q^{-2}(1+kd_{\partial\Omega})^{-2})^{\beta}}\\
&\ge Cq,
\end{aligned}$$
when $k^2q^{-2}(1+k\delta)^{-2}\ge \max_{\overline{\Omega}}(1+2|D\tilde{\phi}|^2)$ and $\delta$ is small enough. 

We first fix the sufficiently big constant $q$ such that $Cq\ge \max u_t$. Thus
$$\frac{\det D^2\overline{w}}{(1+|D\overline{w}|^2)^{\beta}}\ge \frac{\det D^2u}{(1+|Du|^2)^{\beta}}\quad\text{ in }\Omega_{\delta}\times\{t_0\}.$$
Then we take $\delta$ to be small enough and $k\delta\ge 1$ to make
$$k^2q^{-2}(1+k\delta)^{-2}\ge \frac{1}{4}q^{-2}\delta^{-2}\ge \max_{\overline{\Omega}}(1+2|D\tilde{\phi}|^2).$$
Finally, we take $k$ sufficiently large to satisfy 
$$\log (k\delta)\ge q(\max_{\overline{\Omega}}|\tilde{\phi}|+\max_{\overline{Q}_T}|u|),$$ then
$\overline{w}\le u$ on $\partial \Omega_{\delta}\times\{t_0\}$. Therefore $u\ge \overline{w}$ in $\Omega_{\delta}\times\{t_0\}$ by the comparison principle. Then for the exterior normal vector $\nu$ at any point $x_0$ on $\partial\Omega$, we have $u_{\nu}\le w_{\nu}\le C$. By convexity of $u$ in $x$ and arbitrariness of $t_0$, this proves the bound for $|Du|$ on $\overline{Q}_T$.

\textit{(b) $C^2$ estimates on the boundary}.
We can prove this by the same method as before but with different barrier functions. As for the second tangential derivatives, we can also regard $u_t(1+|Du|^2)^{\frac{n+1}{2}}$ which has positive lower and upper bounds as a data function. Then it follows that
\begin{equation}\label{4-uxx}
C^{-1}I_{n-1}\le D_{x'x'}u\le CI_{n-1}\quad\text{on}\quad \partial_p Q_T.
\end{equation}

To prove the uniform bound for the mixed second derivatives, we use the comparison principle for $\mathcal{L}'$. Denote $\mathcal{T}$ as in~\eqref{t-alpha}. Then
\begin{equation}\label{2-c0}
  |\mathcal{L}'\mathcal{T}(u-\phi)|\le \underline a\Big(1+\sum_{i=1}^n u^{ii} \Big)
\end{equation}
for some positive constant $\underline a$ depending on $||\phi||_{\tilde{C}^4(\overline{Q}_T)}$, the positive lower and upper bounds for $u_t$ and the bound for $|Du|$.

Recall that $0\in\partial\Omega$ and the boundary $\partial\Omega$ can be represented as in~\eqref{boundary} near $0$.
On $\Omega\times\{0\}$, $\mathcal{T}(u-\phi)=0$
and
\begin{equation}\label{p-o}
  |\mathcal{T}(u-\phi)|\le \bar a|x'|^2,\quad\text{on $(B_{\varepsilon}\cap \partial\Omega)\times[0,T]$}.
\end{equation}

In $(B_\varepsilon\cap\Omega)\times[0,T]$ with $B_\varepsilon=\{|x|<\varepsilon\}$, we shall consider the following barrier function
  $$\zeta(x,t)=\frac{1}{2}\sum_{\alpha,\beta<n}(B_{\alpha\beta}-\mu\delta_{\alpha\beta}) x_{\alpha}x_{\beta}+\frac{1}{2}Mx_n^2-Nx_n,$$
where $\mu$ is a fixed small constant such that $B_{\alpha\beta}-\mu\delta_{\alpha\beta}>0$, and $M, N$ will be determined later.

First we take $\varepsilon$ small, $M\varepsilon\le 2$ and $N$ larger than $1+\lambda_{max}/c$, where $\lambda_{max}$ is the maximum eigenvalue of the $(n-1)\times(n-1)$ matrix $(B_{\alpha\beta} -\mu\delta_{\alpha\beta})$, and $c$ is the constant such that $x_n\ge c|x'|^2$ on $\partial\Omega$. Then we have $\zeta<0$ in $(B_\varepsilon\cap\Omega)\times\{0\}$.

Thus
\begin{equation}\label{z-b}
  A\zeta \pm \mathcal{T}(u-\phi)\le 0,\quad  \text{in $(B_\varepsilon\cap\Omega)\times\{0\}$ for all positive constant $A$.}
\end{equation}

Note that we can choose a large $M\ge 2$ and a small $c_1>0$ such that
\begin{equation}\label{l'z}
\begin{split}
  \mathcal{L}'\zeta&\ge c_1 \Big(\sum_{i=1}^n u^{ii} + Mu^{nn} \Big)-c_2 \\
  &\ge \frac12 c_1 \sum_{i=1}^n u^{ii} -c_2+ c_3 M^{1/n}. 
  \end{split}
\end{equation}
  Here we use the bounds for $u_t$ and the fact that
  \begin{equation*}
  \begin{split}
  \frac1n \Big(\sum_{i=1}^n u^{ii} + Mu^{nn}\Big) \ge \frac1n \Big(\sum_{i=1}^{n-1} \lambda_i^{-1} + M\lambda_n^{-1}\Big)\ge \Big(M\lambda_1^{-1} \cdots \lambda_n^{-1}\Big)^{1/n},
  \end{split}
  \end{equation*}
  where $\lambda_1\le \cdots \le \lambda_n$ are the eigenvalues of $\{u_{ij}\}$. $c_3$ is a positive constant with $c_3\le Cc_1$, where $C$ depends on $n$ and the maximum of $u_t$. The constant $c_2$ depends on the maximum of $u_t$, $|Du|$ and the diameter of $\Omega$.
  Hence, for $A\ge\max\{1,2\underline a /c_1\}$ and $M\ge ((c_1+c_2)/c_3)^n$, we have
  \begin{equation}\label{l-pm}
  \mathcal{L}'(A\zeta \pm \mathcal{T}(u-\phi)) \ge0, \quad\text{in}\quad (B_\varepsilon\cap\Omega)\times(0,T].
  \end{equation}
  From now on $M$ is fixed.

  Next, we estimate $\zeta$ on  $\partial(B_\varepsilon\cap\Omega)\times[0,T]$.
  For $\varepsilon$ small enough, we have
  \begin{equation}\label{z-s1}
  \zeta \le -\frac12\mu\sum_{\beta<n}x_\beta^2 + O(|x'|^3) \le -\frac14 \mu|x'|^2,\quad\text{on $(\overline{B_\varepsilon}\cap\partial\Omega)\times[0,T]$.}
  \end{equation}
  Then, on $(\partial B_\varepsilon\cap\Omega)\times[0,T]$ and if $\frac14 \mu|x'|^2 > Mx_n^2$, we have
  \begin{equation}\label{z-s2}
  \begin{split}
  \zeta &\le \frac{1}{2}(B_{\alpha\beta}-\mu\delta_{\alpha\beta}) x_{\alpha}x_{\beta}+\frac{1}{2}Mx_n^2-x_n \\
  &\le -\frac{1}{2}\mu|x'|^2 + C|x'|^3 + Mx_n^2 \\
  &\le -c_4 \varepsilon^2,\quad\quad \text{for $\varepsilon$  small enough.}
  \end{split}
  \end{equation}
  Additionally, on $(\partial B_\varepsilon\cap\Omega)\times[0,T]$ and if $\frac14 \mu|x'|^2 \le Mx_n^2$, we have $x_n\ge c_5\varepsilon$ and
  \begin{equation}\label{z-s3}
  \begin{split}
  \zeta &\le  C|x'|^2 +Mx_n^2-x_n \\
  &\le  (M+C)\varepsilon^2 -c_5\varepsilon \\
  &\le -\frac12c_5 \varepsilon,\quad\quad \text{for $\varepsilon$  small enough.}
  \end{split}
  \end{equation}
  Since $|\mathcal{T}(u-\phi)|$ has a upper bound in $\overline{Q}_T$, which supposed to be $c_6$, then by~\eqref{p-o},~\eqref{z-s1},~\eqref{z-s2} and~\eqref{z-s3}, we get
  \begin{equation}\label{z-s}
  A\zeta \pm \mathcal{T}(u-\phi) \le0, \quad\text{on}\quad \partial(B_\varepsilon\cap\Omega)\times[0,T],
  \end{equation}
  for $A=\max\{1,2\underline a/c_1,4\bar a/\mu,c_6/(c_4\varepsilon^2),2c_6/(c_5\varepsilon) \}$.
  Thus by maximum principle and equations~\eqref{l-pm},~\eqref{z-b} and~\eqref{z-s}, we get
  \begin{equation}
  A\zeta \pm \mathcal{T}(u-\phi) \le0, \quad\text{in}\quad (B_\varepsilon\cap\Omega)\times[0,T].
  \end{equation}

  Since $\zeta(0,t_0)=\mathcal{T}(u-\phi)(0,t_0)=0$ for each $t_0\in[0,T]$, we have
  $$|\partial_n\mathcal{T}(u-\phi)(0,t_0)|=|u_{\alpha n}(0,t_0)-\phi_{\alpha n}(0,t_0)|\le AN,$$
  which yields that
  \begin{equation}
    |u_{\alpha n}(0,t_0)|\le AN+||\phi||_{C^2(\overline{Q}_T)}\le C.
  \end{equation}

By the equation~\eqref{gauss} and estimate~\eqref{4-uxx}, we then get $u_{nn}\le C$ on $\partial\Omega\times[0,T]$ for some universal constant $C$. Combining with the positive lower and upper bounds of $u_t$, we have
\begin{equation}\label{g-u-p}
1/C\le \lambda(D^2u)\le C\quad\text{on }\partial_pQ_T
\end{equation}
for some positive constant $C$ depending on $n,\Omega,c_0, ||\phi||_{\tilde{C}^4(\overline{Q}_T)}$ and the uniform convexity of $\phi(\cdot,0)$.

\textit{(c) global $C^2$ estimates}. The Legendre transform $U$ of $u$ satisfies
$$-U_t\det D^2U=(1+|y|^2)^{-\beta}\quad\text{in }Q_T^*$$
with 
$C^{-1}\le -U_t\le C $ and $||D^2U||_{L^{\infty}(\partial_pQ_T^*)}\le C$. Conducting the same calculation as in Lemma~\ref{asl},  we have
$$\tilde{L}U_{kk}\ge -2\beta,$$
where $\tilde{L}=\frac{1}{U_t}\partial_t+U^{ij}\partial_{ij}$.
Thus $\tilde{L}(U_{kk}-Ct)\ge 0$ for some universal constant $C$. Hence by maximum principle and~\eqref{g-u-p},
$$\sup_{Q_T^*}U_{kk} \le C\Big(\sup_{\partial_pQ_T^*}U_{kk}+T\Big)\le C,$$
which yields $\lambda(D^2u)\ge C$ in $Q_T$ for some positive constant $C$ depending on $n,\Omega,c_0, T, ||\phi||_{\tilde{C}^4(\overline{Q}_T)}$ and uniform convexity of $\phi(\cdot,0)$. Therefore we have proved
$$1/C\le \lambda(D^2u)\le C\quad\text{on }\overline{Q}_T.$$

\textit{(d) $\tilde{C}^{2,\alpha}$ estimates}. The interior higher regularity of $u$ follows from that of $U$ in $Q_T^*$ by~\cite{tso}. Under the same condition of Theorem~\ref{ut}, it holds $||u_t||_{\tilde{C}^{\alpha}(\overline{Q}_T)}\le C$ by Krylov-Safonov's estimate.

Note that if we let $v=T_k(u-\phi)$, $1\le k\le n-1$,  as in subsection~\ref{1.3}, we have
$$-v_t+a^{ij}v_{ij}+b^iv_i=f\quad\text{in }G=B_{R_0}^+\times(t_0-R_0^2,t_0].$$
The term $b^iv_i$ comes from the differential of $(1+|Du|^2)^{-\beta}$ and the related coordinate transformation. Here $a^{ij}$ is positive definite, and $|b^i|$, $|v_i|$ and $|f|$ are bounded due to the boundedness of $u_t$, $Du$ and $D^2u$.
This is the case we have dealt with in equation~\eqref{v-eq}. Thus the $\tilde{C}^{2,\alpha}$ estimate of $u$ on $\overline{\Omega}\times[\epsilon,T]$ follows from the same argument in the proof of Corollary~\ref{up}.

The short-time existence theorem ensures the $\tilde{C}^{2,\alpha}$ estimate of $u$ in a small time interval. The proof in this case is similar to the proof of Theorem~\ref{se}. The Fr\'echet differential of the map $\Phi$, defined similarly as~\eqref{phi-df}, is a linear parabolic operator with extra lower order terms. It is still an isomorphism under some suitable compatibility conditions. Thus the short-time existence of the solution follows from the inverse map theorem.

In conclusion, we have proved the global $\tilde{C}^{2,\alpha}$ estimate for $u$, the solution of equation~\eqref{gauss}.
\end{proof}

We also present an existence theorem for equation~\eqref{gauss}.

\begin{thm}
Assume  $\Omega$ is a smooth and uniformly convex domain in $\mathbb{R}^n$, $\phi$ is a smooth function in $\overline{Q}_T$ with $\phi_t\ge c_0 >0$ on $\partial\Omega\times[0,T]$, and $\phi(\cdot,0)$ is uniformly convex. If equation~\eqref{gauss} satisfies the compatibility condition of any order, then there exists a unique smooth solution $u$ to the initial boundary value problem~\eqref{gauss} that is convex in $x$.
\end{thm}

\begin{rem}
Using the above method, one can also prove the global $\tilde{C}^{2,\alpha}$ estimate for the solutions to the first initial-boundary value problem of $\gamma$-Gauss curvature flow as well as the existence of solutions when $0<\gamma\le 1$.
\end{rem}

\section*{Acknowledgement}

We would like to thank Professor X.-J. Wang for his suggestions on this topic.

\end{document}